\documentclass{amsart}
\usepackage[T1]{fontenc}
\newtheorem{theorem}[subsection]{Theorem}

\title{Actions of a group of prime order without equivariantly simple germs}
\author{Ivan Proskurnin}
\date{}

\begin{document}
\maketitle

\begin{abstract}

We prove that equivariantly simple invariant singularities can only exist for very few representations of a group of prime order: for real representations and some ``almost, but not quite real'' representations.
\end{abstract}

\section{Introduction}

One of the most famous results in singularity theory is the classification of simple singularities obtained by V. I. Arnold in 1972 (see \cite{1}). It turns out that the simple singularities are in one-to-one correspondence with Dynkin diagrams with no double or triple edges (ADE-diagrams).  There are multiple connections between simple singularities and other objects admitting an ADE-classification (quadratic forms, Coxeter groups, etc), and the nature of these connections is still quite mysterious. Equivariantly simple singularities are invariant (under some group action) function germs simple with respect to equivariant changes of coordinates. It was also V.I. Arnold (\cite{2}) who discovered that equivariantly simple singularities are similarly connected to Dynkin diagrams --- he demonstrated that $\mathbb{Z}_2$-invariant singularities for certain group actions are in bijection with Dynkin diagrams with double edges. Recently obtained classifications of $\mathbb{Z}_2$ and $\mathbb{Z}_3$-invariant equivariantly simple singularities (\cite{3,4}) show a similar correspondence.

\medskip

One of the challenges that arise when we attempt to calculate a list of simple singularities for a certain group action is figuring out if simple singularities exist at all. In some simple scenarios the nonexistence of simple singularities can be proven by a dimension counting argument. E.g., let $\mathbb{Z}_k$ act on  $\mathbb{C}^n$ by multiplying all the variables by the same primitive root of unity of degree $k$. Then the dimension of the space of homogeneous invariants of the smallest degree (i.e. the homogeneous polynomials of degree $k$) is $C_{n+k-1}^k$ and the dimension of the group of linear changes of variables acting on it is $n^2$. Since for $k\geq4, n\geq2$ the former value is greater than the latter, the invariant functions in this case inevitably have moduli. However, for more complicated representations dimension counting quickly becomes infeasible.

\medskip

In this paper we prove a theorem describing the actions a finite group of prime order for which equivariantly simple function germs may exist.

\begin{theorem} Let  $\tau$ be a linear action of $\mathbb{Z}_p$ on $(\mathbb{C}^n,0)$, $rk(\tau)$ the maximal rank of  a $\tau$-invariant quadratic form, $p$ -- a prime number. Germs equivariantly simple with respect to $\tau$ may only exist in one of the two cases:

1)$det(\tau) \neq 1, n - rk(\tau) \leq log_2 (p+1)$;

2)$det(\tau) =1, n - rk(\tau) \leq log_2 (2p-1)$.

\end{theorem}

Since any smooth action of a finite group is linearizable, the assumption that the group action is linear is not limiting: we will work with linear actions, however all the results can be generalized to analytic group actions in an obvious fashion.
For a real group action $n-rk(\tau)=0$, and in the case of a real action simple singularities exist (an invariant nondegenerate quadratic form has a Morse singularity, which is simple).

If one looks at a list of $\mathbb{Z}_3$-invariant simple singularities obtained in \cite{4}, one can easily see that the simple singuarities only exist in cases described in the theorem.

\section{Definitions, notation  and auxiliary statements}
The letter $p$ will denote a prime number.
In a suitable local system of coordinates any analytic finite group action on $(\mathbb{C}^n,0)$ can be linearized (see \cite{5}). So we wil consider all group actions to be induced by a linear representation of $G$.
Furthermore, any linear representation of a finite abelian group is diagonalazable, hence further we will usually consider only actions of the form $g \circ (x_1,\ldots ,x_n) \longrightarrow (\chi_1(g) x_1,\ldots , \chi_n(g) x_n)$ with $\chi_i$ being group characters of $\mathbb{Z}_p$.

For a function germ $f$ $J_f$ will denote the \textbf{Jacobian ideal} $<\frac{\partial f}{\partial x_1}, \ldots, \frac{\partial f}{\partial x_n}>$.
$Q_f$ is the Jacobian algebra $\mathbb{C}\{x_1, \ldots, x_n\} / J_f$. $Q_f$ always admits a monomial basis, i.e. the basis consisting of monomials.

If $f$ is invariant with respect to an action of a group $G$, the Jacobian ideal is invariant and therefore $G$ also acts on $Q_f$ by changes of coordinates: $g \circ h(x) = h ( g^{-1} (x))$. The representation of $G$ on $Q_f$ is called \textbf{the equvariant Milnor number} of $f$ and denoted $\mu_G(f)$. $\nu(f)$ is the multiplicity of the trivial representation in $\mu_G(f)$, i.e. the dimension of the space of invariant elements of $Q_f$. This subspace of invariant elements is denoted $Q^G_f$.

If $f=x_1^2 + \ldots + x_k^2 + h(x_{k+1}, \ldots, x_n)$ $Q_f$ if isomorphic to $Q_h$. If $f$ is invariant with respect to some group action, they are isomorphic as group representations. In particular, $\mu(f) = \mu(h)$, $\nu(f) = \nu(h)$.

A \textbf{Milnor fiber} of a germ $f$ is the set $V_f = f^{-1} (\epsilon) \cap B_{\delta}, 0<|\epsilon| \ll \delta, \epsilon \; \in \mathbb{C}, \delta \; \in \mathbb{R}$. It is known that (see, e.g.,  \cite{6}) if the Milnor of $f$ is finite $V_f$ has the homotopy type of a wedge sum of $\mu(f)$ $n-1$-dimensional spheres. In particular, $V_f$ only has nontrivial homology groups in dimension $n-1$, $rk(H_{n-1}(V_f)) = \mu$.

A group action on $(\mathbb{C}^n, 0)$  is called \textbf{real} if it admits an invariant quadratic form of rank $n$.

\section{Equivariantly simple and equivariantly stable germs}

There are multiple papers on the classification of equivariantly simple singularities (see., e.g., \cite{2}, \cite{3}, \cite{4}, \cite{7}). Roughly speaking, an \textbf{equivariantly stable} singularity is the simplest of the equivariantly simple singularities, i.e. the one that is only adjacent to itself. If the case of real group action the equivariantly stable singularity is the nondegenerate one, as the deformation of a Morse function is again a Morse function. For the formal definitions of equivariant simplicity and stability see below.

We will denote $\mathfrak{m}_G$ the maximal ideal in the ring of $G$-invariant analytic germs. Here we only assume G to be finite, not nessesarily abelian or isomorphic to $\mathbb{Z}_p$. Consider $\mathcal{O}^{r}_{G}$ --- the space of $r$-jets of elements in $\mathfrak{m}_G$ with a critical point at the origin.  The group $D^r_G$ of $r$-jets of  $G$-equivariant diffeomorphisms  $(\mathbb{C}^n,0) \longrightarrow (\mathbb{C}^n,0)$ acts on this space.

A $G$-invariant germ $f:(\mathbb{C}^n,0) \longrightarrow (\mathbb{C},0)$ with a critical point at the origin is \textbf{equivariantly simple} if there is an $N \; \in \; \mathbb{N}$ such that for all $r$ large enough  $r$-jet of $f$ has a neighbourhood in $\mathcal{O}^{r}_{G}$ that only intersects at most $N$ $D^r_G$-orbits.

A $G$-invariant germ $f:(\mathbb{C}^n,0) \longrightarrow (\mathbb{C},0)$ with a critical point at the origin is \textbf{equivariantly stable} if for all large enough $r$  $D^r_G$-orbit of $r$-jet of $f$ in  $\mathcal{O}^{r}_{G}$ is open.

\begin{theorem}
$f$ is equivariantly stable iff $\nu(f)=1$.

\end{theorem}
\begin{proof}
Let $f$ be equivariantly stable. As the orbit of $r$-jet of $f$ is open for large enough $r$ the tangent space to the orbit has dimension equal to the dimension of $\mathcal{O}^{r}_{G}$ itself. The tangent space to the orbit  at $r$-jet of $f$ is the projection of the jacobian ideal of $f$ to  $\mathcal{O}^{r}_{G}$. But if for all $r$ large enough the projection of $J_f$ covers all  $\mathcal{O}^{r}_{G}$ $J_f$ contains $\mathfrak{m}_G$. Since $f$ has a critical point at the origin  $J_f$ does not contain 1 and $\nu(f)=1$.

Now assume $\nu(f)=1$. The Milnor number of $f$ is finite, as the invariants of a finite group action have the separating property, i.e. for any two orbits of the $G$-action there is an invariant function that is equal to 1 on the first orbit and 0 on the second. But if $\nu(f)=1$, then $J_f = \mathfrak{m}_G$ and any invariant function takes the same value at any point of the singular locus of $f$ as it does at the origin. So the separating property fails if the critical point of $f$ is non-isolated. Since  $f$ has a finite Milnor number we can apply the Slodowy theorem (see \cite{8}), according to which the invariant versal deformation of $f$ is equal to $f + \lambda_1 h_1 + \ldots + \lambda_{\nu} h_{\nu}$, $h_j$ --- linearly independent invariant elements of $Q_f$. Since $\nu(f)=1$, the versal deformation of $f$ is trivial, i.e. has the form $f_{\lambda} = f + \lambda$, and the equivariant stability follows immediately.
\end{proof}
In particular for the nondegenerate singularity $\mu(f)=\nu(f)=1$, as $Q_f$ is equal to $\mathbb{C}$.

\begin{theorem}The existence of equivariantly simple singularities for a given group action is equivalent to the existence of equivariantly stable singularities.
\end{theorem}

\begin{proof}
In one direction this statement is self-evident as stable singularities are by definition simple.

Let $f$ be an equivariantly simple singularity. Choose $r_0 \geq |G|$ such that a neighbourhood of $r_0$-jet of $f$ intersects with only a finite number of $D^{r_0}_G$-orbits. Each of the orbits is a submerged manifold, and since a finite number of orbits cover a neighbourhood of $j^{r_0} (f)$ one of the orbits must be open. Let $g$ be the germ with an open orbit. The tangent space to the orbit of $j^{r_0} (g)$ has the maximal dimension (equal to the dimension of $\mathcal{O}^{r_0}_{G}$). Since the tangent space is the projection of the Jacobian ideal of $g$ to  $\mathcal{O}^{r_0}_{G}$ for each invariant polynomial $q$ of degree $\leq r_0$ there is an $\tilde{q} \; \in \; J_g$ with $j^{r_0} (\tilde{q}) = q$. According to Noether's theorem (see, e.g., \cite{9}) the ring of invariant polynomials can be generated by homogeneous polynomials of degree at most $|G|$,  so for each $r>r_0$ and any invariant polynomial $q$ of degree  $r$ there is also a $\tilde{q} \; \in \; J_g$ with $j^{r} (\tilde{q}) = q$, i.e. the orbit of $g$ is open for all $r>r_0$.

\end{proof}

Therefore we can prove the nonexistence of equivariantly simple singularities by proving the nonexistence of invariant germs with $\nu = 1$. This is precisely what will be done.

\section{Real group actions and Roberts' (in)equality}

The next two theorems are taken from \cite{10}.

\begin{theorem}
If $f$ is invariant with respect to a real action of a finite group $G$ then there is a $G$-invariant deformation $f_{\lambda}$ of $f$ with only nondegenerate critical points.
\end{theorem}

The set $C$ of critical points of any invariant deformation $f_{\lambda}$ is $G$-invariant. The action of $G$ on $(\mathbb{C}^n, 0)$ induces a representation  $\theta$ of $G$ on the space of $\mathbb{C}$-valued functions on $C$ ($\theta(g) \circ h(x) = h ( g^{-1} (x))$.

\begin{theorem}If the deformation $f_{\lambda}$ has only nondegenerate critical points then the representations $\theta$ and  $\mu_G(f)$ are isomorphic.
\end{theorem}

The last statement mean in particular that $\nu(f)$ is equal to the number of orbits of critical points in a invariant Morse approximation of $f$. 

Assume that a real $G$-action $\tau$ has no fixed points outside of origin, and denote $l(\tau)$ the smallest length of an orbit of a point outside of origin. By counting the critical points of an invariant morsification of $f$ we obtain the \textbf{Roberts' inequality} $\mu(f) \geq (\nu(f) -1) l(\tau) + 1$. If $G=\mathbb{Z}_p$ each orbit outside the origin has length $p$, so we in fact have an equality
\begin{equation}{\nonumber}
 \mu(f) = (\nu(f) -1) p + 1.
 \end{equation}
 We will refer to this equation as \textbf{Roberts' equality}.

\section{Real representation corresponding to an arbitrary representation}

For a group action $ \tau: g \circ (x_1,\ldots ,x_n) \longrightarrow (\chi_1 (g) x_1,\ldots , \chi_n (g) x_n)$ we will denote $\tau_{\mathbb{R}}$ the action on $\mathbb{C}^{2n}$ defined as  $g \circ (x_1,\ldots ,x_n, y_1, \ldots, y_n) \longrightarrow \\ (\chi_1 (g) x_1,\ldots , \chi_n (g) x_n, \bar{\chi}_1 (g) y_1,\ldots , \bar{\chi}_n (g) y_n)$. 

$\tau_{\mathbb{R}}$ is always a real action with the invariant quadratic form $\Sigma_i x_i y_i$. If $\tau$ has no fixed points outside of origin then neither does $\tau_{\mathbb{R}}$.

For a function $f:(\mathbb{C}^n,0) \longrightarrow (\mathbb{C},0)$ $f \oplus f$ is the function $f(x_1, \ldots, x_n) + f(y_1, \ldots, y_n)$.

 It is easy to see that if $f$ has an isolated singularity at the origin then so does $f \oplus f$, and $\mu(f \oplus f) = (\mu(f))^2$. Furthermore, if $f$ is invariant with respect to an action $\tau$ then $f \oplus f$ is invariant with respect to the action $\tau_{\mathbb{R}}$.

\begin{theorem} If $p_1, \ldots, p_{\mu}$ form a monomial basis for the Jacobian algebra of $f$ then $p_i(x) p_j (y), 1 \leq i,j \leq \mu$ form a monomial basis in the Jacobian algebra of  $f \oplus f$.
\end{theorem}
\begin{proof} Start with an arbitrary monomial basis ${q_1(x,y), \ldots , q_{\mu^2}(x,y)}$ for $Q_{f \oplus f}$. Each $q_i$ factors as as a product of monomials in $x$ and $y$: $q_i = q_{i1}(x) q_{i2}(y)$. For each $i$ $q_{i1} \notin J_f$, as otherwise $q_i \in J_{f \oplus f}$. Hence we have a decomposition $q_{i1} = \Sigma_{j}c_{i j} p_j (x)$. Similarly, each $q_{i2}$ is a linear combination of $p_j(y)$. Therefore $p_i(x) p_j (y)$ form a system of generators for $Q_{f \oplus f}$. As there are exactly $\mu^2$ of them, they form a basis.

\end{proof}

\begin{theorem}$\nu(f \oplus f) = k_1^2 + \ldots + k_m^2$, with $k_i$ being the multiplicities of different one-dimensional representations in the irreducible decomposition of $Q_{f }$.
\end{theorem}
\begin{proof} As proven in the previous theorem $Q_{f \oplus f}$ is isomorphic ( as a vector space) to the tensor product of  $Q_f$ with itself. On the second factor of this tensor product (i.e. on monomials in $y_i$) $G$ acts by conjugate characters, so the action on the group on the second factor can be identified with the action on the conjugate space $Q^{*}_f$. According to a well-known result in representation theory (see, e.g., \cite{11}, pp. 433--434) for each finite-dimensional representation $V$ the representation $V \otimes V^{*}$is naturally isomorphic to  $Hom(V, V)$, and the invariant subspace of  $V \otimes V^{*}$  is isomorphic to the space of equivariant endomorphisms of $V$. The dimension of the space of equivariant endomorphisms is $k^2_1 + \ldots + k^2_m$ with  $k_i$ being the multiplicities of irreducible representations in the decomposition of $V$.

\end{proof}
 
 \section{Equivariant Milnor number of a $\mathbb{Z}_p$-invariant singularity}
 
Consider the Milnor fiber $V_f$ of a $G$-invariant function $f$ with an isolated critical point at $0$. $G$-action on $\mathbb{C}^n$  induces a linear action $\psi_G(f)$ on the homology group $H_{n-1}(V_f)$. Wall (\cite{12}) has proven that

\begin{equation}{\nonumber}
\psi_G(f) = (det(\tau))^{-1} \otimes \mu_G(f).
 \end{equation}{\nonumber}

Let $W$ denote the restriction of the regular representation of $\mathbb{Z}_p$ to $\{x_1 + \ldots + x_p=0\}$ and $\tilde{\mathbb{C}}$ denote the trivial one-dimensional representation. Then (see \cite{13}) for any $\mathbb{Z}_p$-invariant function germ $f$ with an isolated critical point  $\psi_{\mathbb{Z}_p}(f) = a \; \tilde{\mathbb{C}} + b \; W$ with some nonnegative $a,b \: \in \; \mathbb{Z}$ 

Combining this with Wall's equation we obtain $\mu_{\mathbb{Z}_p}(f) = a \; det (\tau) + b \; (W \otimes det(\tau))$. 
Since all irreducible representations except the trivial one occur in $W$ with multiplicity 1, the multiplicities of one-dimensional representations in $\mu_{\mathbb{Z}_p}(f)$ are $a$ for $det (\tau)$ and $b$ -- for all the remaining one-dimensional representations of $\mathbb{Z}_p$. 

\section{Proof of main theorem}

Let $f$ be a germ equivariantly stable with respect to a linear $\mathbb{Z}_p$-action $\tau$ for some prime p. 
Terms of degree 2 in the Taylor series of a stable germ must add up to a quadratic form of maximal rank  $rk(\tau)$. According to the splitting lemma in some local system of coordinates  $f$ is equal to $x_1^2 + \ldots + x_{rk(\tau)}^2 + h(x_{rk(\tau)+1}, \ldots, x_n)$, with $j^2 (h)=0$. Then $\mu(f) = \mu(h)$, $\mu(h) \geq 2^{n-rk(\tau)}$ (as $h$ has no terms of degree 2 in its Taylor series).  The subspace $\{x_1= \ldots = x_{rk(\tau)} = 0\}$ is invariant (since it is the kernel of an invariant bilinear form). $\mathbb{Z_p}$ acts on this subspace with no fixed points: a square of any invariant variable is an invariant function, and therefore an invariant bilinear form of maximal rank has no invariant vectors in its kernel.

\medskip

Consider the germ $h \oplus h$ as defined in section 5.  The germ $h \oplus h$ is invariant with respect to a real action with no fixed points outside of origin, so the Roberts' equality can be applied to it: $\mu(h \oplus h) = (\nu(h \oplus h) - 1)p + 1$. $\nu(h \oplus h)$ is equal to $a^2 + b^2 (p-1)$  for some $a,b$ according to sections 5 and 6, so
$\mu(h \oplus h) = \mu^2 (h) = (a^2 - 1)p + b^2 p (p-1) +1$. 

\medskip

On the other hand, using the equation  $\mu_{\mathbb{Z}_p}(h) = a \; det (\tau) + b \; (W \otimes det (\tau))$ the Milnor number of $h$ can be immediately calculated as $a + b(p-1)$ and therefore $\mu^2(h) = a^2 + 2 a b (p-1) + b^2 (p^2 - 2p +1)$.

\medskip

Since the Jacobian algebras of $f$ and $h$ are isomorphic $\nu(h) = \nu(f) = 1$. As $\mu_{\mathbb{Z}_p}(h) = a \; det (\tau) + b \; (W \otimes det (\tau))$, either $det(\tau) \neq 1, b=1$, or  $det(\tau) = 1, a=1$. Let us consider both of these possibilities.

\medskip

1)Substituting $b=1$ into previously obtained equations for $\mu(h \oplus h)$ yields $\mu^2 (h) =  (a^2 - 1)p + p (p-1) +1 = a^2 + 2 a (p-1) + (p^2 - 2p +1)$. After some trivial manipulations we obtain $a^2 (p-1) - 2a(p-1) = 0$, i.e. $a(a-2)(p-1)=0$. Therefore $a=0$ or $a=2$. The former yields $\mu(h)=\mu(f) = p-1$, the latter $\mu(h)=\mu(f) = p+1$.

2)Similarly we substitute $a=1$  and obtain $b^2 p (p-1) +1 = 1 + 2 b (p-1) + b^2 (p^2 - 2p +1)$. This simplifies to $b^2 (p-1) - 2b(p-1)=0$, i.e. $b=0$ or $b=2$. Therefore $\mu(f) = 1$ (in this case $f$ is a Morse germ) or $\mu(f) = 2p-1$.

\medskip

To sum up these calculations, the Milnor number of a gem $f$ equivariantly stable with respect to a $\mathbb{Z}_p$-action $\tau$ is at most $p+1$ if $det(\tau) \neq 1$ and at most $2p-1$ if $det(\tau) = 1$. Since $\mu(f) = \mu(h) \geq 2^{n-rk(\tau)}$ this concludes the proof of the main theorem.

\end{document}